\documentclass[final]{elsarticle}

\usepackage{amssymb}
\usepackage{amsmath}
\usepackage{amsthm}
\usepackage[scr=boondoxo]{mathalfa}
\usepackage{xcolor}    
\usepackage{upgreek}    
\usepackage{parskip}    
\usepackage{mathtools}    
\usepackage[margin=1in]{geometry}
\usepackage{cleveref}

\crefformat{equation}{(#2#1#3)}

\renewcommand{\vec}[1]{{\mathchoice
                     {\mbox{\boldmath$\displaystyle{#1}$}}
                     {\mbox{\boldmath$\textstyle{#1}$}}
                     {\mbox{\boldmath$\scriptstyle{#1}$}}
                     {\mbox{\boldmath$\scriptscriptstyle{#1}$}}}}

\newcommand{\avv}{\alpha_{vv}}
\newcommand{\aww}{\alpha_{ww}}
\newcommand{\avw}{\alpha_{vw}}

\newcommand{\mat}[1]{\mathsf{\mathbf{#1}}}
\newcommand{\mip}[2]{{\langle{#1}, {#2} \rangle}_\hilb}

\newcommand{\GM}[2]{\mathcal{N}\left( {#1}, {#2}\right)}

\newcommand{\obs}{\vec{y}} 
\newcommand{\prcov}{\mat{\Gamma}_{\mathrm{pr}}}
\newcommand{\postcov}{\mat{\Gamma}_{\mathrm{post}}}
\newcommand{\noise}{\mat{\Gamma}_{\!\mathrm{n}}}

\newcommand{\ave}[2]{\mathbb{E}_{{#1}}\!\left\{ {#2} \right\}}

\newcommand{\R}{\mathbb{R}}

\newcommand{\F}{\mat{F}}

\newcommand{\HM}{\mat{H}}
\newcommand{\HMt}{\tilde{\mat{H}}}

\newcommand{\hilb}{\mathscr{H}}

\newcommand{\dparpr}{\dpar_{\mathrm{pr}}}
\newcommand{\dparpost}{\dpar^\obs_{\mathrm{post}}}
\newcommand{\priorm}{\upmu_\mathrm{pr}}
\newcommand{\postm}{\upmu_{\mathrm{post}}^{\obs}}
\renewcommand{\Xi}{X^{-1}}
\newcommand{\dpar}{\vec{m}}

\newcommand{\avey}[1]{\mathbb{E}_{{\vec{y}|\dpar}}\left\{ {#1} \right\}}
\newcommand{\DKL}{\mathrm{D}_{\mathrm{kl}}}
\newcommand{\tran}{{\mkern-1.5mu\mathsf{T}}}                

\newcommand{\ff}{\vec{f}}
\newcommand{\fft}{\vec{\tilde{f}}}
\newcommand{\EIG}{\Phi_{\mathrm{eig}}}
\newcommand{\defeq}{\vcentcolon=}
\newcommand{\deriv}[3]{\Delta_{{#1}}({#2}\,|\,{#3})}
\newcommand{\Ns}{n_s}

\makeatletter

\newtheorem{theorem}{Theorem}[section]%
\ifx\@NOAMS\undefined\theoremstyle{ejpecpbodyrm}\fi%
\newtheorem{remark}[theorem]{Remark}%

\begin{document}

\begin{frontmatter}

\title{On submodularity of the expected information gain}

\author[ncsu]{Steven Maio}
\ead{smaio@ncsu.edu}
\author[ncsu]{Alen Alexanderian}
\ead{alexanderian@ncsu.edu}

\affiliation[ncsu]{organization={North Carolina State University},
            city={Raleigh},
            state={NC},
            country={USA}}

\begin{abstract}
We consider finite-dimensional linear Gaussian Bayesian inverse problems with
uncorrelated sensor measurements.  In this setting, it is known that the
expected information gain, quantified by the expected Kullback--Leibler
divergence from the posterior measure to the prior measure, is submodular.  We present a simple
alternative proof of this fact tailored to a weighted inner product space
setting arising from discretization of infinite-dimensional 
inverse problems constrained by partial differential equations (PDEs). 
\end{abstract}

\begin{keyword}
    Bayesian inversion \sep Information gain \sep Optimal experimental design \sep Submodularity
\end{keyword}

\end{frontmatter}

\section{Introduction}
We consider the estimation of a parameter $\dpar \in \R^n$ from the observation model
\begin{equation}\label{equ:linmodel}
    \obs = \F \mkern1mu  \dpar + \vec{\eta}.
\end{equation}
Here, $\obs \in \R^q$ is a data vector, $\F$ is a linear transformation, 
and $\vec\eta \sim \GM{\vec{0}}{\noise}$ 
is a random vector that models measurement noise. 
We assume that $\vec{\eta}$ and $\dpar$ are independent.
Also, we assume the measurements are uncorrelated and therefore, 
$\noise = \mathrm{diag}(\sigma_1^2, \ldots, \sigma_q^2)$.
We consider the case where $\F$ is a linear transformation arising from
discretization of an infinite-dimensional linear inverse 
problem~\cite{bui2013computational}. In inverse problems governed by PDEs, 
applying $\F$ to a vector requires solving the
governing PDEs and extracting solution data at a set of measurement locations.  
The parameter space is $\hilb \defeq \R^n$ endowed with the 
inner product 
$\mip{\vec{u}}{\vec{v}} \defeq \vec{u}^\tran \mat{M} \vec{v}$, for $\vec u$ and $\vec v$ 
in $\hilb$, where $\mat{M}$ is a symmetric positive definite matrix. In the case 
of finite element discretization of inverse problems, $\mat{M}$ is the 
mass matrix~\cite{bui2013computational}.
Additionally, 
we assume a Gaussian prior, $\priorm = \GM{\dparpr}{\prcov}$, 
where
$\prcov$ is a strictly positive selfadjoint operator on $\hilb$.
The solution of the present linear Gaussian 
inverse problem is the Gaussian posterior, $\postm = \GM{\dparpost}{\postcov}$, with~\cite{Stuart10}
\begin{equation}\label{equ:postmeas}
    \dparpost = \postcov(\F^*\noise^{-1}\obs + \prcov^{-1}\dparpr)
\quad
\text{and}
\quad
    \postcov = (\HM + \prcov^{-1})^{-1},
\end{equation}
where $\HM = \F^* \noise^{-1} \F$. Note that here $\F$ is a linear transformation, 
$\F:\hilb \to \R^q$ and its adjoint $\F^*$ is given by $\F^* = \mat{M}^{-1}\F^\tran$;
see~\cite{bui2013computational} for details.

Data acquisition is a crucial consideration when solving inverse problems. 
We consider problems in which data are obtained from sensor measurements and
examine
the problem of
finding sensor placements that provide the most informative data.
This is an 
optimal experimental design (OED)~\cite{AtkinsonDonev92,Ucinski05} problem.
For a review of literature on OED for inverse problems governed by PDEs, 
see~\cite{Alexanderian21}.
A popular approach in this context is to seek measurements that 
maximize the expected information gain (EIG) defined as 
\begin{equation}\label{equ:avgKL} 
\text{EIG} \defeq  
\ave{\priorm}{\avey{\DKL(\postm\| \priorm)}}. 
\end{equation} 
Here, $\DKL$
denotes the Kullback--Leibler divergence~\cite{kullback:1951,ManningSchutze99} from the
posterior measure to the prior measure.
Also, $\mathbb{E}_{\priorm}$ and $\mathbb{E}_{{\vec{y}|\dpar}}$, respectively, denote expectation with respect to 
$\priorm$ and the likelihood measure, $\mu_{\obs | \dpar} \defeq \GM{\F \dpar}{\noise}$.
In the present linear Gaussian setting, EIG 
admits the analytic expression $\frac12 \log\det(\mat{I} + \HMt)$,
where $\HMt$ is the prior-preconditioned operator,
$\HMt = \prcov^{1/2}\HM\prcov^{1/2}$. For further details, 
see, e.g.,~\cite{AlexanderianGloorGhattas16,AlexanderianSaibaba18}. 

Let us denote, 
\begin{equation}\label{equ:EIG}
\EIG := \log\det(\mat{I} + \HMt).
\end{equation} 
We focus on the OED problem of finding an optimal subset of a 
prescribed cardinality from a set of $\Ns$ candidate
sensor locations such that $\EIG$ is maximized.
To make matters concrete, 
we consider
inverse problems governed by stationary PDEs where the
number $q$ of rows of $\F$ is equal to $\Ns$. 
The present setting covers the case of inverse problems governed by time-dependent 
PDEs in which sensor measurements are collected at the final time as well.

Subsequently, the OED problem can 
be formulated as
finding an optimal subset of the rows of $\F$.
Specifically, let
$V = \{1, \ldots, \Ns\}$ index the set of candidate sensor locations. Given $S
\subseteq V$, we denote by $\F(S)$ the matrix obtained by selecting rows of $\F$
indexed by $S$. This way, we can state $\EIG$ as a function of $S$. 
We thus consider the optimization problem
\begin{equation}\label{equ:OED}
\max_{\{S \subseteq V : |S|\leq k\}} \EIG(S),
\end{equation}
where $k \leq \Ns$ is a prescribed sensor budget.
This is a challenging optimization problem with combinatorial complexity. 
An exhaustive search for an optimal subset is impractical for even modest values of 
$\Ns$ and $k$. For example, with $\Ns = 80$ and $k = 10$, an exhaustive 
search for the solution of~\cref{equ:OED} requires 
${80\choose 10} = \mathcal{O}(10^{12})$ function evaluations. A popular approach 
for obtaining a near optimal solution is to follow a greedy strategy,
where sensors are chosen one at a time; at each step, we choose a sensor that 
results in the largest incremental increase in the value of $\EIG$. 
It is known that $\EIG$ considered above is a monotone submodular set function. 
This has been noted in various contexts; see,
e.g.,~\cite{KelmansKimelfeld83,GuestrinKrauseKrause05,KrauseSinghGuestrin08,ShamaiahBanerjeeVikalo10}
This observation enables drawing 
certain theoretical guarantees~\cite{NemhauserWolseyFisher78,KrauseGolovin14} 
regarding sensor placements obtained by 
applying a greedy procedure to the problem~\cref{equ:OED}.
Namely, if $S_k$ is obtained from 
the greedy algorithm, the following holds 
\begin{equation}\label{equ:greedy_thm}
\EIG(S_k) \geq (1-1/e) \max_{\{S \subseteq V : |S|\leq k\}} \EIG(S),
\end{equation}
where $e$ is the base of the natural logarithm. 

In this brief note, we provide a simple proof of monotonicity and submodularity
of $\EIG$.  The aim is to provide additional insight regarding the properties of
$\EIG$ for the PDE-based linear  problems under study.

\section{Preliminaries}\label{sec:preliminaries}

\subsection{Some linear algebra tools}
In what follows, we need to manipulate 
rank-one updates of invertible operators.  
Recall that 
the tensor product of $\vec{u}$ and $\vec{v}$ in $\hilb$, denoted 
by $\vec{u} \otimes \vec{v}$, is the rank-one linear operator satisfying
\[
(\vec{u} \otimes \vec{v})\vec{x}= \mip{\vec{v}}{\vec{x}}\vec{u}, \quad \vec{x} \in \hilb.
\]
With the present choice of inner product, 
$\vec{u} \otimes \vec{v} = \vec{u} \vec{v}^\tran \mat{M}$. 

For an invertible linear operator $\mat{A}$ on $\hilb$, we have
\begin{align}
    \det(\mat{A} + \vec{u} \otimes \vec{v}) &= \big(1 + \mip{\mat{A}^{-1}\vec{u}}{\vec{v}}\big) \det(\mat{A})\label{equ:Auov},\\
    \vspace{1mm}
    \left(\mat{A}+\vec{u} \otimes\vec{v} \right)^{-1} &= 
        \mat{A}^{-1} - \frac{\mat{A}^{-1}(\vec{u} \otimes \vec{v}) \mat{A}^{-1}}
        {1+\mip{\mat{A}^{-1}\vec{u}}{\vec{v}}}.\label{equ:smw}
\end{align}
Note that~\cref{equ:Auov} follows from the well-known formula for the determinant of 
a rank-one perturbation of a nonsingular matrix. And the identity~\cref{equ:smw}
follows from the Sherman--Morrison--Woodbury identity.
The latter assumes $1+\mip{\mat{A}^{-1}\vec{u}}{\vec{v}} \neq 0$. 

\subsection{Submodular functions}
Consider a finite set $V = \{1, \ldots, \Ns\}$ and a set function $f: 2^{V}\to
\mathbb{R}$. 
We first recall 
the notion of the marginal gain, which is also 
called the discrete derivative~\cite{KrauseGolovin14}. 
For $S \subseteq V$ and $v \in V$, we define the marginal 
gain of $f$ at $S$ with respect to $v$ by 
\[
\deriv{f}{S}{v} \defeq f(S \cup \{ v\}) - f(S).
\]
For example, in the sensor placement problem described in the introduction, 
$\deriv{\EIG}{S}{v}$ is the marginal benefit, in terms of expected information
gain, of taking a measurement  at the $v$th candidate sensor location. 

We say that $f$ is 
\begin{itemize}
\item \textit{monotone}, if for every $A\subseteq B
\subseteq V$, $f(A) \le f(B)$;
\item  
\textit{submodular}, if for every $A \subseteq B \subseteq V$ and $v \in V\backslash B$,
\begin{equation}\label{equ:submodularity}
    \deriv{f}{A}{v} \ge \deriv{f}{B}{v}.
\end{equation}
\end{itemize}
Note that Submodularity captures the idea of diminishing
returns~\cite{KrauseGolovin14}.  Proving submodularity can be reduced to the
case of one element increments~\cite[Theorem 44.1]{Schrijver03}. Namely, to prove $f$ is submodular, 
we need to verify that for every 
$A \subseteq V$ and distinct $v, w\in V\backslash A$
\begin{equation}\label{equ:incr-submodularity}
    \deriv{f}{A}{v} 
    \ge 
    \deriv{f}{A\cup\{w\}}{v}. 
\end{equation}

\section{The main result}
First, we consider a representation of $\EIG$ that 
facilitates understanding its properties. 
Let $\{\vec{e}_i\}_{i=1}^{\Ns}$ be the standard basis in $\R^{\Ns}$, 
and consider 
the operator $\HM$ in~\cref{equ:postmeas}. Note that
\begin{equation*}
\HM = \F^* \noise^{-1}\F 
= \F^* \Big(\sum_{i=1}^{\Ns} \sigma_i^{-2} \vec{e}_i \vec{e}_i^\tran\Big) \F 
= \sum_{i=1}^{\Ns} (\sigma_i^{-1}\F^*\vec{e}_i) (\sigma_i^{-1}\mat{M}^{-1}\F^\tran\vec{e}_i)^\tran\mat{M} 
= \sum_{i=1}^{\Ns} (\sigma_i^{-1}\F^*\vec{e}_i) (\sigma_i^{-1}\F^*\vec{e}_i)^\tran\mat{M}.
\end{equation*}
Thus, letting $\vec{f}_i := \sigma_i^{-1}\F^*\vec{e}_i$, 
we have $\HM = \sum_{i=1}^{\Ns} \ff_i \otimes \ff_i$.
Note that the $i$th term in this summation corresponds to
the $i$th candidate sensor location. 
This problem structure is a result 
of the assumption of uncorrelated measurements and has been noted 
in various contexts; see, e.g.,~\cite[Chapter 3]{Ucinski05} or 
the review~\cite{Alexanderian21}. 

We next note that the prior-preconditioned operator $\HMt$ in~\cref{equ:EIG}
admits the following convenient form:
\begin{equation}\label{equ:HMt_rep}
    \HMt = \sum_{i=1}^{\Ns} \fft_i \otimes \fft_i,
\end{equation}
where $\fft_i = \prcov^{1/2}\ff_i$, for $i \in \{1, \ldots, \Ns\}$. 
We can thus formulate the problem of finding an optimal
subset of candidate sensor locations as that of identifying an optimal subset of 
the terms in that summation.  
Specifically,
for a set $S \subseteq V \defeq \{1, \ldots, \Ns\}$, we define 
\[
    \HMt(S) \defeq \sum_{i \in S} \fft_i \otimes \fft_i.
\]
Also, for $S = \emptyset$, we let $\HMt$ be the zero matrix. 
Subsequently, we consider the problem of finding an optimal subset 
$S \subseteq V$ of cardinality $k \leq \Ns$ that maximizes, 
\begin{equation}\label{equ:EIG_S}
\EIG(S) = \log\det(\mat{I} + \HMt(S)).  
\end{equation} 

Since a zero row in $\F$ has no impact on 
$\EIG$, 
in what follows, we assume $\ff_i \neq \vec{0}$, 
for $i \in \{1, \ldots, \Ns\}$.
We are now ready to prove that $\EIG(S)$ is a monotone submodular function. 

\clearpage
\begin{theorem}\label{prp:monotone}
The following hold:
\begin{enumerate}[(i)]
\item $\EIG(\emptyset) = 0$;
\item $\EIG$ is strictly monotone; and
\item $\EIG$ is submodular.
\end{enumerate}
\end{theorem}
\begin{proof}
The first statement follows immediately from definition of $\EIG$. 
In the rest of the proof we let $S$ be an arbitrary subset of $V$ and define 
\[
\mat{A} \defeq \mat{I} + \HMt(S).
\] 
Moreover, 
for $i, j \in V$, we denote 
$\alpha_{ij} = \mip{\mat{A}^{-1}\fft_i}{\fft_j}$.

To prove the strict monotonicity, it is sufficient to show that 
for every $v \in V \setminus S$, 
$\EIG(S \cup \{v\}) > \EIG(S)$.
To see this, we first 
use~\cref{equ:Auov} to note 
\begin{equation}\label{equ:avv}
    \det(\mat{A} + \fft_v \otimes \fft_v) = 
    \big(1 + \mip{\mat{A}^{-1}\fft_v}{\fft_v} \big) \det(\mat{A})
    = (1+\avv)\det(\mat{A}). 
\end{equation}
Also, since $\mat{A}^{-1}$ is strictly positive and $\fft_v \neq 0$, we know 
$\avv > 0$.
Subsequently,
\[
\begin{aligned}
    \EIG(S \cup \{v\}) - \EIG(S)
    &= \log\det(\mat{I} + \HMt(S) + \fft_v \otimes \fft_v) - \log\det(\mat{I} + \HMt(S))\\
    &= \log\det(\mat{A} + \fft_v \otimes \fft_v) - \log\det(\mat{A})\\
    &= \log\big( \det(\mat{A} + \fft_v \otimes \fft_v)  \det(\mat{A})^{-1} \big) \\
    &= \log(1+\avv) > 0.
\end{aligned}
\]

Finally, we show $\EIG$ is submodular.
Let $v$ and $w$ be distinct elements of $V\setminus S$. As we noted above, 
\[
\deriv{\EIG}{S}{v} = 
\EIG(S \cup \{v\}) - \EIG(S)
= \log\big( 1 + \avv). 
\]
We next note that 
\begin{equation}\label{equ:expand_ss}
        \det(\mat{A} +\fft_v \otimes \fft_v+ \fft_w \otimes \fft_w )= 
                                                                      \big(1 + \mip{(\mat{A} + \fft_w \otimes \fft_w)^{-1}\fft_v}{\fft_v} \big) \det(\mat{A} + \fft_w \otimes \fft_w). 
\end{equation}
Now, by~\cref{equ:smw},
\begin{equation}\label{equ:expand_smw}
\left(\mat{A}+\fft_w \otimes\fft_w \right)^{-1}
    =\mat{A}^{-1} - \frac{(\mat{A}^{-1}\fft_w) \otimes (\mat{A}^{-1}\fft_w)} 
    {1+\mip{\mat{A}^{-1}\fft_w}{\fft_w}}.
\end{equation}
This implies
$\mip{(\mat{A} + \fft_w \otimes \fft_w)^{-1}\fft_v}{\fft_v}
= \avv - {\avw^2}/(1+ \aww)$. Note also that by strict positivity of 
$(\mat{A} + \fft_w \otimes \fft_w)^{-1}$, we know $\avv - {\avw^2}/(1+ \aww) > 0$. 
Subsequently,
\[
    \det(\mat{A} + \fft_v \otimes \fft_v + \fft_w \otimes \fft_w)
    = 
    (1 + \avv - \avw^2/(1+\aww))\det(\mat{A} + \fft_w \otimes \fft_w).
\]
Thus, we have 
\[
\begin{aligned}
    \deriv{\EIG}{S \cup\{w\}}{v} = 
    \EIG(S \cup&\{ v, w\})-\EIG(S \cup \{ w\}) \\
        &= \log\det(\mat{A} + \fft_v \otimes \fft_v + \fft_w \otimes \fft_w ) - 
           \log\det(\mat{A} + \fft_w \otimes \fft_w)\\
        &= 
        \log\big[ (1 + \avv - \avw^2/(1+\aww))\det(\mat{A} + \fft_w \otimes \fft_w) \det(\mat{A} + \fft_w \otimes \fft_w)^{-1}\big]
        \\
        &= \log(1 + \avv - \avw^2/(1+\aww)).
\end{aligned}
\]
Therefore, since $\avw^2/(1+\aww) \geq 0$, 
\[
    \deriv{\EIG}{S \cup\{w\}}{v} = 
    \log(1 + \avv - \avw^2/(1+\aww))  
   \leq \log(1 + \avv)  = \deriv{\EIG}{S}{v}. \qedhere
\]
\end{proof}
\begin{remark}
A key to showing submodularity of $\EIG$ was 
to state the prior-preconditioned operator $\HMt$ as a sum of 
rank-one operators. This in turn was possible, due to 
assumption of uncorrelated measurements. 
\end{remark}

\bibliographystyle{elsarticle-num}
\bibliography{refs.bib}

\end{document}